\documentclass[12pt,a4]{amsart}
\usepackage[a4paper, left=28mm, right=28mm, top=28mm, bottom=34mm]{geometry}
\usepackage[all]{xy}
\usepackage{amsmath}
\usepackage{amssymb}
\usepackage{amsthm}
\usepackage{url}
\usepackage{mathrsfs}
\theoremstyle{definition}
\newtheorem{thm}{Theorem}[section]

\newtheorem{cor}[thm]{Corollary}
\newtheorem{prop}[thm]{Proposition}

\theoremstyle{definition}

\newtheorem{rem}[thm]{Remark}

\newtheorem{dfn}[thm]{Definition}

\newtheorem{ex}[thm]{Example}

\newcommand{\bD}{\mathbb{D}}
\newcommand{\bH}{\mathbb{H}}
\newcommand{\bN}{\mathbb{N}}
\newcommand{\bZ}{\mathbb{Z}}
\newcommand{\bF}{\mathbb{F}}

\newcommand{\sG}{\mathscr{G}}
\newcommand{\sL}{\mathscr{L}}
\newcommand{\sX}{\mathscr{X}}
\newcommand{\Br}{\mathrm{Br}}
\newcommand{\crys}{\mathrm{crys}}
\newcommand{\Crys}{\mathrm{Crys}}
\newcommand{\dR}{\mathrm{dR}}

\newcommand{\Fr}{\mathrm{Fr}}
\newcommand{\id}{\mathrm{id}}
\newcommand{\im}{\mathrm{Im}}

\newcommand{\Lie}{\mathrm{Lie}}
\newcommand{\Spec}{\mathrm{Spec}}
\newcommand{\Spf}{\mathrm{Spf}}
\newcommand{\NilCrys}{\mathrm{NilCrys}}
\newcommand{\Pic}{\mathrm{Pic}}
\newcommand{\Rad}{\mathrm{Rad}}
\begin{document}

\title[quasi-canonical liftings of $K3$ surfaces of finite height]{Construction of quasi-canonical liftings of $K3$ surfaces of finite height in odd characteristic}

\author{Kentaro Inoue}
\address{Department of Mathematics, Faculty of Science, Kyoto University, Kyoto 606-8502, Japan}
\email{inoue.kentarou.73c@st.kyoto-u.ac.jp}

\maketitle

\begin{abstract}
We construct a quasi-canonical lifting of a $K3$ surface of finite height over a finite field of characteristic $p\geq3$.
Such results are previously obtained by Nygaard-Ogus when $p\geq5$.
For this purpose, we use the display-theoretic deformation theory developed by Langer, Zink, and Lau. 
We study the display structure of the crystalline cohomology of deformations of a $K3$ surface of finite height in terms of the Dieudonn\'e display of the enlarged formal Brauer group.
\end{abstract}

\section{Introduction}
The notion of quasi-canonical liftings of varieties over a perfect field $k$ of positive characteristic is introduced by Nygaard-Ogus to show that Tate's conjecture holds for $K3$ surfaces of finite height (\cite[Definition 1.5]{no}). Let $X_0$ be a $K3$ surface of finite height over a finite field $k$ of characteristic $p$. They proved that if there is a quasi-canonical lifting of $X_{0}$, the $K3$ surface $X_{0}$ satisfies Tate's conjecture \cite[Theorem 2.1 and Remark 2.2.3]{no}. Moreover, they constructed a quasi-canonical lifting of $X_{0}$ under the assumption that $p\geq 5$. 

In this paper, we prove the following theorem:

\begin{thm}[see Theorem \ref{4.1}] \label{1.1}
Let $X_{0}$ be a $K3$ surface of finite height $h<\infty$ over a finite field $k$ of odd characteristic $p$. Then there exists a totally ramified finite extension $V/W(k)$ of degree $h$ and a quasi-canonical lifting $X/V$ of $X_{0}/k$.
\end{thm}

This result is previously obtained by Nygaard-Ogus when $p\geq5$ (see \cite[Theorem 5.6]{no}). We give an alternative proof when $p\geq5$. Our result seems new when $p=3$.

Recall that Nygaard-Ogus established a one-to-one correspondence between deformations of $X_{0}$ over $k[t]/(t^n)$ and deformations of the crystalline cohomology $H^{2}_{\crys}(X_{0}/W(k))$ over $k[t]/(t^n)$ as crystals with additional structure (Frobenius, pairing, and Hodge filtration) under the assumption that $p\geq5$. We need the assumption that $p\geq5$ to justify some calculation for divided power (see \cite[Lemma 4.6]{no}). Furthermore, this correspondence is proved only over the base $k[t]/(t^n)$ because crystals behave well only over a base whose PD envelope is $p$-torsion free (see the proof of \cite[Theorem 4.5]{no}).

In this paper, we use the display-theoretic deformation theory of $K3$ surfaces developed by Langer, Zink, and Lau (\cite{lz2}, \cite{lau2}). When $p \geq 3$, the crystalline cohomology of a $K3$ surface is naturally equipped with a display structure, and there is a one-to-one correspondence between deformations of $X_0$ over an arbitrary Artin local ring $R$ with residue field $k$ and deformations of the crystalline cohomology $H^{2}_{\crys}(X_{0}/W(k))$ over $R$ as displays with additional structure.

The key step to prove Theorem \ref{1.1} is to study a relation between the display associated to the crystalline cohomology of a $K3$ surface and the Dieudonn\'e display of the enlarged formal Brauer group. We note that the following Theorem is a display-theoretic analogue of \cite[Theorem 3.20]{no} (for the notation on displays, see Section \ref{Section:2}).

\begin{thm}[see Corollary \ref{3.5}]
Let $k$ be a perfect field of characteristic $p\geq3$, $R$ be an Artin local ring with residue field $k$, and $X$ be a $K3$ surface of finite height over $R$ (i.e.\ a proper flat scheme over $R$ whose closed fiber is a $K3$ surface of finite height). Let $\widehat{\Br}_{X/R}$ (resp.\ $\psi_{X/R}$) be the formal Brauer group (resp.\ the enlarged formal Brauer group) associated to $X/R$. Then there exists the following exact sequence of displays over the small Witt frame $\widehat{W}(R)$:
\begin{align*}
 0\to \underline{D}(\psi_{X/R})\to \underline{H}^{2}_{\crys}(X/\widehat{W}(R))\to \underline{D}(\widehat{\Br}^{\ast}_{X/R})(-1)\to 0.   
\end{align*} 
Moreover, this sequence is compatible with base change with respect to $R$.
\end{thm}

\begin{rem}
From \cite[Theorem 2.1]{no} and Theorem \ref{1.1}, it follows that Tate's conjecture holds for $K3$ surfaces of finite height over finite fields of characteristic $p\geq3$.
Note that Tate's conjecture for $K3$ surfaces is previously proved by Madapusi Pera
\cite[Theorem 1]{mp}, \cite[Theorem A.1]{km}.
(See also \cite{mp2}, \cite[Section 6.4]{iik}.)
\end{rem}

\begin{rem}
If an extension of the base field $k$ is allowed, Theorem \ref{1.1} is previously obtained by Ito-Ito-Koshikawa \cite[Corollary 9.11]{iik} using the \'etaleness of the Kuga-Satake morphism. However, it seems difficult to control the extension degree by the method of \cite{iik}.
\end{rem}

\begin{rem}
In characteristic $2$, it is not known whether the crystalline cohomology of a $K3$ surface is equipped with a display structure.
This is why we assume $p \geq 3$ throughout this paper.
\end{rem}

The organization of this paper is as follows. In Section \ref{Section:2}, we review the display-theoretic deformation theory of $p$-divisible groups and $K3$-surfaces. In Section \ref{Section:3}, we study the relation between the Dieudonn\'e display of the enlarged formal Brauer group and the display of a $K3$ surface. In Section \ref{Section:4}, we prove the main theorem of this paper.

\section{Frames and displays}
\label{Section:2}

In this section, we review some results about displays following \cite{lau2} (see also \cite{lz1}, \cite{lz2}).

\subsection{Frames}

\begin{dfn}[{\cite[Definition 2.0.1]{lau2}}]
A \emph{frame} $\underline{S}=(S, \sigma, \tau)$ consists of a commutative $\bZ$-graded ring 
\begin{align*}
 S=\bigoplus_{n\in \bZ} S_n   
\end{align*}
and ring homomorphisms $\sigma, \tau: S\to S_0$ satisfying the following conditions:
\begin{enumerate}
    \item$\tau_0: S_0\to S_0$ is the identity, and $\tau_{-n}:S_{-n}\to S_0$ is bijective for $n\geq1$. Let $t\in S_{-1}$ be the unique element such that $\tau_{-1}(t)=1$.
    \item $\sigma_0: S_0\to S_0$ is a Frobenius lift (i.e.\ a ring homomorphism satisfying $\sigma_0(x)-x\in pS_0$ for all $x\in S_0$), and $\sigma_{-1}(t)=p$.
    \item $p\in \Rad(S_0)$. Here $\Rad(S_0)$ is the Jacobson radical of $S_0$.
\end{enumerate}

 A \emph{morphism of frames} is a morphism of graded rings that commutes with $\sigma$ and $\tau$.
\end{dfn}

\begin{rem} \label{2.2}
A frame $\underline{S}$ is uniquely determined by the graded ring $S_{\geq0}=\bigoplus_{n\geq0}S_n$ together with the ring homomorphism $\sigma:S_{\geq0}\to S_0$ and the homomorphism of graded $S_{\geq0}$-modules $t: S_{\geq1}\to S_{\geq0}$ satisfying the following conditions:
\begin{enumerate}
    \item $\sigma_0: S_0\to S_0$ is a Frobenius lift, and $\sigma(t(a))=p\sigma(a)$ for all $a\in S_{\geq1}$.
    \item $p\in \mathrm{Rad}(S_0)$.
\end{enumerate}
\end{rem}

\begin{ex} \label{2.3}
Let $A$ be a $p$-torsion free ring with $p\in \mathrm{Rad}(A)$, and $\sigma_0: A\to A$ be a Frobenius lift. There is a unique frame $\underline{A}=(A[t], \sigma, \tau)\ (\deg(t)=-1)$ such that $\sigma_{-1}(t)=p$ and $\tau_{-1}(t)=1$. This frame is called the \emph{tautological frame} associated to $A$.
\end{ex}

\begin{ex} \label{2.4}
Let $R$ be an Artin local ring with residue field of characteristic $p\geq3$. Let $\widehat{W}(R)$ be the small Witt ring defined in \cite[Section 2]{zi1}. We obtain a frame $(\underline{\widehat{W}}(R), \sigma, \tau)$ as follows. We set $\underline{\widehat{W}}(R)_0=\widehat{W}(R)$ and $\underline{\widehat{W}}(R)_n=\hat{I}(R)=V(\widehat{W}(R))$ as a $\underline{\widehat{W}}(R)_0$-module for $n\geq1$. For $n,m\geq1$, a multiplication map $\underline{\widehat{W}}(R)_n\times \underline{\widehat{W}}(R)_m\to \underline{\widehat{W}}(R)_{n+m}$ is given by 
\begin{align*}
 \underline{\hat{I}}(R)\times \underline{\hat{I}}(R)\to \underline{\hat{I}}(R), \quad (V(a), V(b))\mapsto V(ab).   
\end{align*}
Let $\sigma_0:\widehat{W}(R)\to \widehat{W}(R)$ be the Witt vector Frobenius, and we put $\sigma_n(V(a))=a$ for $a\in \widehat{W}(R)$ and $n\geq1$. The map $t:\underline{\widehat{W}}(R)_1\to \underline{\widehat{W}}(R)_0$ is the inclusion map, and $t:\underline{\widehat{W}}(R)_{n+1}\to \underline{\widehat{W}}(R)_n$ is the multiplication map $p:\hat{I}(R)\to \hat{I}(R)$ for $n\geq1$. This determines a frame $(\underline{\widehat{W}}(R), \sigma, \tau)$ uniquely by Remark \ref{2.2}. This frame is denoted simply by $\underline{\widehat{W}}(R)$ and called the \emph{small Witt frame} associated to $R$.
\end{ex}

\begin{rem} \label{2.5}
Let $A$ be as in Example \ref{2.3}, $R$ be as in Example \ref{2.4}, and $f: A\to \widehat{W}(R)$ be a ring homomorphism. Then $f$ induces the unique morphism of graded rings $A[t]\to \underline{\widehat{W}}(R)$ which commutes with $\tau$. This map commutes with $\sigma$, and induces a morphism of frames $\underline{A}\to \underline{\widehat{W}}(R)$.
\end{rem}

\subsection{Displays}

\begin{dfn}[{\cite[Definition 3.2.1]{lau2}}]
Let $\underline{S}=(S,\sigma,\tau)$ be a frame.
Let $\pi: S \to S_0/(tS_1)$ be the composite of
the projection $S\to S_0$ and $S_0 \to S_0/(tS_1)$.
\begin{enumerate}
\item An \emph{$\underline{S}$-display} (or simply \emph{display}) $\underline{M}=(M, F)$ consists of a finite projective graded $S$-module $M$ and a $\sigma$-linear map $F:M\to M^{\tau}$ which induces an isomorphism of $S_0$-modules $M^{\sigma}\stackrel{\sim}{\to} M^{\tau}$. Here, for a ring homomorphism $f:A\to B$ and an $A$-module $N$, we put $N^{f}:=N\otimes_{A,f} B$.
\item A \emph{morphism of displays} is a morphism of graded $S$-modules that commutes with $F$.
\item We call a display $(M, F)$ an \emph{effective display} if $M^{\pi}:=M\otimes_{S,\pi}S_0/tS_1$ concentrates on non-negative degrees.
\item An effective display $(M, F)$ is called a \emph{$d$-display} if $(M^{\pi})_i=0$ for $i\notin [0, d]$. Here $(M^{\pi})_i$ denotes the degree $i$ part of the graded $S_0/tS_1$-module $M^{\pi}$.
\end{enumerate}
\end{dfn}

\begin{ex}
Let $\underline{S}=(S, \sigma, \tau)$ be a frame and $n\in \bZ$. Then $(S(n), \sigma)$ is an $\underline{S}$-display. Here we put $S(n)_i:=S_{n+i}$. This display is denoted by $\underline{S}(n)$.
\end{ex}

\begin{dfn}
Let $\underline{M}=(M, F)$ and $\underline{M'}=(M', F')$ be $\underline{S}$-displays.
\begin{itemize}
    \item The \emph{direct sum} of them is the $\underline{S}$-display $\underline{M}\oplus \underline{M'}=(M\oplus M', F\oplus F')$.
    \item The \emph{tensor product} of them is the $\underline{S}$-display $\underline{M}\otimes_{\underline{S}} \underline{M'}=(M\otimes_{S} M', F\otimes_{S} F')$.
    \item The \emph{dual} of  $\underline{M}$ is the $\underline{S}$-display $\underline{M}^{\ast}=(M^{\ast}, (F^{\ast})^{-1})$. Here $M^{\ast}$ is the dual of the $S$-module $M$.
\end{itemize} 
\end{dfn}

\begin{dfn}
Let $\underline{M}=(M, F)$, $\underline{M'}=(M', F')$, and $\underline{M''}=(M'', F'')$ be $\underline{S}$-displays. A sequence of morphisms of displays
\begin{align*}
 0\to \underline{M}\to \underline{M'}\to \underline{M''}\to 0
 \end{align*} is an \emph{exact sequence of displays} if the underlying sequence of $S$-modules
\begin{align*}
 0\to M\to M'\to M''\to 0   
\end{align*} is exact.
\end{dfn}

\begin{dfn}
Let $\underline{S}\to \underline{S'}$ be a morphism of frames, and $\underline{M}=(M, F)$ be an $\underline{S}$-display. The \emph{base change} of $\underline{M}$ by $\underline{S}\to \underline{S'}$ is the $\underline{S'}$-display $\underline{M}\otimes_{\underline{S}} \underline{S'}=(M\otimes_{S} S', F\otimes_{S} S')$.
\end{dfn}

\begin{rem} \label{taut disp}
Let $A$ be a $p$-adic ring (i.e.\ complete and separated in the $p$-adic topology). Assume that $A$ is $p$-torsion free. Let $\sigma:A\to A$ be a Frobenius lift. An $\underline{A}$-display is equivalent to a finite projective $A$-module $M$ with a descending finite filtration by direct summands $(F^iM)_{i\in \bZ}$ together with $\sigma$-linear maps $F_i: F^iM\to M$ satisfying the following conditions (\cite[Example 3.6.2]{lau2}):
\begin{itemize}
    \item $pF_{i+1}=F_i|_{F^{i+1}M}$ for all $i\in \bZ$.
    \item $M$ is generated by the union of all $F_i(F^iM)$.
\end{itemize} 
In this equivalence, a $d$-display corresponds to a triple
\[ (M, (F^iM)_{i\in \bZ}, F_i) \]
such that $F^iM=0$ for $i\geq d+1$ and $F^iM=M$ for $i\leq 0$.
\end{rem}

\begin{rem} \label{2.12}
Let $R$ be as in Example \ref{2.4}. An effective $\underline{\widehat{W}}(R)$-display is equivalent to a display $(P_i, \iota_i, \alpha_i, F_i)_{i\geq0}$ over the small Witt frame in \cite[Definition 3]{lz2} (see also \cite[Example 3.6.3]{lau2}). In this equivalence, a $d$-display corresponds to $(P_i, \iota_i, \alpha_i, F_i)_{i\geq0}$ such that $\alpha_i: \hat{I}(R)\otimes P_i\to P_{i+1}$ is surjective (equivalently, $(P_i, \iota_i, \alpha_i, F_i)_{i\geq0}$ has a \emph{normal decomposition} $(L_i)_{i\geq0}$ in the sense of \cite[Definition 3]{lz2} such that $L_i=0$ for $i\geq d+1$). In particular, a $1$-display over $\underline{\widehat{W}}(R)$ is equivalent to a quadruple 
\begin{align*}
    (P, Q, F, F_1),
\end{align*} where $P$ is a finite projective $\widehat{W}(R)$-module, $Q$ is a submodule of $P$ such that $\hat{I}(R)P\subset Q$, the quotient $P/Q$ is projective over $R$, the maps $F:P\to P$ and $F_1:Q\to P$ are $\sigma$-linear morphisms such that $F_1(Q)$ generates $P$ as a $\widehat{W}(R)$-module. We note that this is a \emph{Dieudonn\'e display} over $R$ in the sense of \cite[Definition 2.6]{lau1}.
\end{rem}

\subsection{Displays and $p$-divisible groups over $p$-torsion free bases with Frobenius lifts}

Let $A$, $p$, and $\sigma$ be as in Remark \ref{taut disp}. Moreover, we assume that $p\geq3$. Let $G$ be a $p$-divisible group over $A$. We have a covariant Dieudonn\'e crystal $\bD(G)$ which is a crystal on the nilpotent crystalline site $\NilCrys(\Spf(A)/\Spf(\bZ_p))$ (\cite[Chapter IV]{me}). We set
\begin{align*}
    D(G)&:= \bD(G)_{A}
\\ F^1D(G)&:= \mathrm{Ker}(\bD(G)_{A}\twoheadrightarrow \mathrm{Lie}(G))=\mathrm{Lie}(G^{\ast})^{\ast}.
\end{align*}
Let $G_0:=G\otimes A/pA$. The ring morphism $\sigma$, the Frobenius map $\Fr:A/pA\to A/pA$, and the Verschiebung map  $G_{0}^{(p)}\to G_0$ induce the following morphisms of $A$-modules
\begin{align*} D(G)^{\sigma}&=(\bD(G_0)_{(A\twoheadrightarrow A/pA)})^{\sigma}
\\ &=(\Fr_{\crys}^{\ast}\bD(G_0))_{(A\twoheadrightarrow A/pA)}
\\ &=\bD(G_{0}^{(p)})_{(A\twoheadrightarrow A/pA)}
\\ &\to \bD(G_0)_{(A\twoheadrightarrow A/pA)}
 \\ &=D(G). 
\end{align*} 
Note that the PD thickening $A\twoheadrightarrow A/pA$ is nilpotent because $p\geq3$. The $\sigma$-linear morphism induced by the composite of the above morphisms is denoted by $F:D(G)\to D(G)$.

The following result is presumably well-known to experts (see \cite[Theorem 3.17]{lau1}). We give the proof for readers' convenience.

\begin{prop} \label{dd}
The triple $(D(G)\supset F^1D(G), F)$ is a $1$-display over the tautological frame $\underline{A}$ (see Remark \ref{taut disp}). 
\end{prop}

\begin{proof}
It suffices to prove that $F(F^1D(G))$ is contained in $pD(G)$, and $D(G)$ is generated by $F(D(G))$ and $p^{-1}F(F^1D(G))$. 

We consider the following $A/pA$-linear map 
\[
F_0: (\bD(G_0)_{A/pA})^{(p)}=D(G)^{\sigma}\otimes A/pA \to D(G)\otimes A/pA =\bD(G_0)_{A/pA}.   
\]
We consider the following diagram.
\[
\xymatrix{
\mathrm{Lie}((G_0^{(p)})^{\ast})^{\ast} \ar@{=}[r] \ar[d] & F^1D(G)^{\sigma} \otimes A/pA \ \ar@{^{(}->}[r] & D(G)^{\sigma}\otimes A/pA \ar[d]^{F_0}
\\ \mathrm{Lie}(G_0^{\ast})^{\ast} \ar@{=}[r] & F^1D(G)\otimes A/pA \  \ar@{^{(}->}[r] & D(G)\otimes A/pA
}
\]
In this diagram, the left vertical map is zero because this map is induced by the relative Frobenius morphism $G_0^{\ast}\to (G_{0}^{\ast})^{(p)}=(G_0^{(p)})^{\ast}$. Thus $F(F^1D(G))\subset pD(G)$.

Let us fix an $A$-submodule $L$ of $D(G)$ such that $D(G)=L\oplus F^1D(G)$.
The $\sigma$-linear maps $F|_L: L\to D(G)$ and $p^{-1}F: F^1D(G)\to D(G)$ induce a $\sigma$-linear map 
\[
 \widetilde{F}: D(G)=L\oplus F^1D(G)\to D(G).   
\]
It suffices to show that $\widetilde{F}$ induces a surjection of $A$-modules $D(G)^{\sigma}\to D(G)$. By Nakayama's lemma, it suffices to prove that, for an arbitrary field $k$ and an arbitrary ring morphism $A/pA\to k$, the Frobenius-linear map $\widetilde{F}\otimes_{A}k$ induces a surjective morphism of $k$-vector spaces $D(G)^{\sigma}\otimes_A k\to D(G)\otimes_A k$. We may assume that $k$ is a perfect field.
Since $A$ is $p$-torsion free, the Frobenius lift $\sigma : A \to A$
induces a ring homomorphism $A \to W(A)$ by \cite[VII, Proposition 4.12]{lazard}.
Thus the composite of $A\to A/pA\to k$ induces a ring homomorphism $A\to W(k)$ commuting with Frobenius lifts. Taking base change by $A\to W(k)$, it suffices to consider the case where $A=W(k)$ and $\sigma$ is the Witt vector Frobenius.

In this case, $D(G)$ is the covariant Dieudonn\'e module of the $p$-divisible group $G_0$ over $k$. We have a $\sigma^{-1}$-linear map $V:D(G)\to D(G)$ satisfying $V(D(G))=F^{-1}(pD(G))$. Therefore $pL\oplus F^1D(G)$ is contained in $V(D(G))$. Since 
\begin{align*}
\dim_k((pL\oplus F^1D(G))/pD(G))&=\dim_k(F^1D(G)/pF^1D(G))
\\ &=\dim_k((\Lie(G_{0}^{\ast}))^{\ast})
\\ &=\dim G^{\ast} \end{align*}
and
\begin{align*}
\dim_k(V(D(G))/pD(G))&=\dim_k(D(G)/F(D(G)))
\\ &=\dim G^{\ast} \end{align*}(the last equality follows from \cite[Chapter III, Section 9]{de}), we see that $pL\oplus F^1D(G)$ coincides with $V(D(G))$. Therefore, we have
\begin{align*}
 \im(\widetilde{F})&=F(L)+(p^{-1}F)(F^1D(G))
 \\ &=(p^{-1}F)(pL\oplus F^1D(G))
 \\ &=(p^{-1}F)(V(D(G)))
 \\ &=D(G).   
\end{align*}
\end{proof}

\begin{dfn}
The display defined in Proposition \ref{dd} is called the \emph{Dieudonn\'e display} associated to $G$ over the tautological frame $\underline{A}$. It is denoted by $\underline{D}(G)$.
\end{dfn}

\begin{prop} \label{duality over tf}
There exists a functorial isomorphism of displays
\[
  (\underline{D}(G))^{\ast}\stackrel{\sim}{\to} \underline{D}(G^{\ast}).
\]
\end{prop}
\begin{proof}
See \cite[Proposition 5.3.6]{bbm}.
\end{proof}

\subsection{Displays and $p$-divisible groups over Artin local rings}

Let $p\geq3$ and $R$ be as in Example \ref{2.4}. Let $G$ be a $p$-divisible group over $R$. We have a covariant Dieudonn\'e crystal $\bD(G)$ which is a crystal on the nilpotent crystalline site $\NilCrys(\Spec(R)/\Spf(\bZ_p))$. We set
\begin{align*}
    D(G)&:= \bD(G)_{(\widehat{W}(R)\twoheadrightarrow R)}
\\ F^1D(G)&:= \mathrm{Ker}(\bD(G)_{(\widehat{W}(R)\twoheadrightarrow R)}\twoheadrightarrow \bD(G)_R\twoheadrightarrow \Lie(G)).
\end{align*} The Witt vector Frobenius $\sigma: \widehat{W}(R)\to \widehat{W}(R)$, the Frobenius map $\Fr:R/pR\to R/pR$, and the Verschiebung map $(G\otimes R/pR)^{(p)}\to G\otimes R/pR$ induce the following morphisms of $\widehat{W}(R)$-modules: 
\begin{align*} D(G)^{\sigma}&=(\bD(G\otimes R/pR)_{(\widehat{W}(R)\twoheadrightarrow R/pR)})^{\sigma}
\\ &=(\Fr^{\ast}_{\crys}\bD(G\otimes R/pR))_{(\widehat{W}(R)\twoheadrightarrow R/pR)}
\\ &=\bD((G\otimes R/pR)^{(p)})_{(\widehat{W}(R)\twoheadrightarrow R/pR)}
\\ &\to \bD(G\otimes
 R/pR)_{(\widehat{W}(R)\twoheadrightarrow R/pR)}
 \\ &=D(G).
\end{align*} The $\sigma$-linear map induced by the composite of the above morphisms is denoted by $F:D(G)\to D(G)$.

\begin{prop}[Lau] \label{dd witt}
There is a unique $\sigma$-linear morphism $F_1:F^1D(G)\to D(G)$ which is functorial in $R$ and $G$ such that the quadruple
\[
 (D(G), F^1D(G), F, F_1)   
\]
is a $1$-display over the small Witt frame $\underline{\widehat{W}}(R)$ (see Remark \ref{2.12}). 
\end{prop}

\begin{proof}
See \cite[Proposition 3.17]{lau1}.
\end{proof}

\begin{dfn}
The display defined in Proposition \ref{dd witt} is called the \emph{Dieudonn\'e display} associated to $G$ over the small Witt frame $\underline{\widehat{W}}(R)$. It is denoted by $\underline{D}(G)$.
\end{dfn}

\begin{prop} \label{duality over alr}
There exists a functorial isomorphism of $\underline{\widehat{W}}(R)$-displays \begin{align*}
 (\underline{D}(G))^{\ast}\stackrel{\sim}{\to} \underline{D}(G^{\ast}).   
\end{align*}
\end{prop}

\begin{proof}
Let $k$ be the residue field of $R$ and $G_0:=G\otimes_R k$. By \cite[Proposition 3.11]{lau1}, we have the universal deformation of $G_0$
\[
\sG\to \Spf(W(k)[[t_1,\ldots,t_d]])
\]
where $d:=\dim(G_0)\  \cdot\ \dim(G_{0}^{\ast})$. We put $A:=W(k)[[t_1,\ldots,t_d]]$ and $A_n:=A/(t_{1}^n,\ldots,t_{d}^n)$ for $n\geq1$. Let $\sigma:A\to A$ be a Frobenius lift such that $\sigma(t_i)=t_{i}^p$ for all $i$ and such that $\sigma$ coincides with the Witt vector Frobenius on $W(k)$. The Frobenius lift $\sigma$ defines tautological frames $\underline{A}$ and $\underline{A_n}$. We have the unique morphism of local $W(k)$-algebras $A\to R,(t_i\mapsto x_i)$ corresponding to $G/R$. The morphism of local $W(k)$-algebras $A\to \widehat{W}(R)$ which sends $t_i$ to $[x_i]$ for all $i$ induces a morphism of frames $\underline{A}\to \underline{\widehat{W}}(R)$. Here, $[x_i]$ denotes the Teichm\"uller lift of $x_i$. Since $R$ is an Artin local ring, the map $A\to R$ (resp.\ $\underline{A}\to \underline{\widehat{W}}(R)$) factors as $A\twoheadrightarrow A_{m}\to R$ (resp.\ $\underline{A}\to \underline{A_m}\to \underline{\widehat{W}}(R)$) for some $m\geq1$. Then, by taking base change of the isomorphism of displays over $\underline{A_m}$ in Proposition \ref{duality over tf}
\[
\underline{D}(\sG\otimes_A A_m)^{\ast} \stackrel{\sim}{\to} \underline{D}((\sG\otimes_A A_m)^{\ast})
\]
by $\underline{A_m}\to \underline{\widehat{W}}(R)$, we get the desired isomorphism.
\end{proof}

The following theorem is a main theorem of display-theoretic Dieudonn\'e theory proved by Lau (see \cite{lau1} for details).

\begin{thm}[Lau] \label{dieudonne theory}
The functor 
\begin{align*}
  (p\text{-divisible groups over}\ R) \to (1\text{-displays over} \ \underline{\widehat{W}}(R))  
\end{align*}
defined by 
\begin{align*}
 G\mapsto \underline{D}(G)   
\end{align*} gives an equivalence of categories.
\end{thm}

\begin{proof}
See \cite[Theorem 3.19]{lau1}.
\end{proof}

\subsection{Displays and $p$-adic formal schemes over $p$-torsion free bases with Frobenius lifts}
\label{subsection:2.5}
Let $A$ and $\sigma:A\to A$ be as in Remark \ref{taut disp}. Moreover, we assume that $A$ is Noetherian and $p\geq3$. Let $\sX$ be a $p$-adic proper smooth formal scheme over $\Spf(A)$ satisfying the following conditions:
\begin{itemize}
    \item The Hodge-de Rham spectral sequence associated to $\sX/A$
    \begin{align*}
       E_{1}^{i,j}=H^j(\sX, \Omega_{\sX/A}^i)\Rightarrow H^{i+j}_{\dR}(\sX/A) 
    \end{align*}
    degenerates at $E_1$-page. 
    \item For $i,j\geq0$, $H^j(\sX, \Omega_{\sX/A}^i)$ is a finite projective $A$-module. It follows from this that $H^j(\sX, \Omega_{\sX/A}^i)$ is compatible with base change with respect to $A$.
\end{itemize}
We put $X_0:=\sX\otimes A/pA$.

When $\sX$ is projective over $A$, the following result is proved by Langer-Zink (see \cite[Theorem 5.5]{lz1}).

 \begin{prop} \label{str div}
 Let
\[
     H^{n}_{\dR}(\sX/A)\supset F^1H^{n}_{\dR}(\sX/A)\supset F^2H^{n}_{\dR}(\sX/A)\supset \cdots
\]
be the Hodge filtration. Let
\[ F: H^{n}_{\dR}(\sX/A)\to H^{n}_{\dR}(\sX/A) \]
be a $\sigma$-linear map induced by $\sigma$ and the identification $H^{n}_{\dR}(\sX/A)\cong H^{n}_{\crys}(X_0/A)$. We assume $p>n$. Then the triple 
\[
  (H^{n}_{\dR}(\sX/A), F^{\bullet}H^{n}_{\dR}(\sX/A), F)  
\]
is an $n$-display over the tautological frame $\underline{A}$ (see Remark \ref{taut disp}).
\end{prop}
 
 \begin{proof}
 It suffices to prove that $F(F^iH^n_{\dR}(\sX/A))$ is contained in $p^iH^n_{\dR}(\sX/A)$ for $i\geq1$, and $H^n_{\dR}(\sX/A)$ is generated by the union of $(p^{-i}F)(F^iH^n_{\dR}(\sX/A))$ for $i\geq0$. 
 
 The former assertion is proved in \cite[Lemma 5.4]{lz1} when $\sX$ is projective over $A$. In the general case, we use the results of Berthelot-Ogus as follows.
 By \cite[Theorem 8.16 and Theorem 8.20]{bo}, for $i\in [0,n]$, we get the following commutative diagram in the derived category.
 \[ \xymatrix{
 \Omega_{\sX^{\sigma}/A}^{\geq i} \ar@{^{(}->}[r] & \Omega_{\sX^{\sigma}/A}^{\epsilon} \ar[rr] \ar[d] & & \Omega^{\bullet}_{\sX^{\sigma}/A} \ar[d]
 \\ & L\eta\Omega^{\bullet}_{\sX/A} \ar[r] & p^iA\otimes^{\mathbf{L}}_{A} \Omega^{\bullet}_{\sX/A} \ar[r] & \Omega^{\bullet}_{\sX/A}
 }
 \]
 Here $\epsilon: \bN\to \bN$ is defined by $\epsilon(j)=\max\{j-i,0\}$, $\eta: \bN\to \bN$ is defined by $\eta(j)=\epsilon(j)+j$, and $\Omega_{\sX^{\sigma}/A}^{\geq i}$(resp. $\Omega_{\sX^{\sigma}/A}^{\epsilon}$) is the subcomplex of $\Omega_{\sX^{\sigma}/A}^{\bullet}$ whose degree $j$ part is $0$ for $j\leq i-1$ and $\Omega^j_{\sX^{\sigma}/A}$ for $j\geq i$ (resp. $0$ for $j\leq -1$ and $\Omega_{\sX^{\sigma}/A}^{\epsilon(j)}$ for $j\geq0$). Taking hypercohomology, the first assertion is proved.
 
 We prove the latter assertion. By the same argument as the proof of Proposition \ref{dd}, it suffices to consider the case that $A=W(k)$ where $k$ is a perfect field and $\sigma$ is the Witt vector Frobenius. Then the latter assertion follows from the strong divisibility of the crystalline cohomology $H^n_{\crys}(X_0/W(k))$ (see \cite[Proposition in Section 1]{fo}).
 \end{proof}
 
 The $n$-display given by Proposition \ref{str div} is denoted by $\underline{H}^n_{\dR}(\sX/A)$.

\subsection{Displays and $K3$ surfaces over Artin local rings}
Let $R$ be an Artin local ring with perfect residue field $k$ of characteristic $p\geq3$.

\begin{dfn}[{\cite[Definition 8.0.1]{lau2}}]
A \emph{K3 display} $(\underline{M},(,))$ consists of a $\underline{\widehat{W}}(R)$-display $\underline{M}=(M,F)$ and a symmetric perfect pairing 
\[
 (,):\underline{M}\times \underline{M}\to \underline{\widehat{W}}(R)(-2)
\] (i.e.\ a morphism of $\underline{\widehat{W}}(R)$-displays $\underline{M}\otimes \underline{M}\to \underline{\widehat{W}}(R)(-2)$ which is a symmetric perfect pairing of graded $\underline{\widehat{W}}(R)$-modules) such that 
\[
  M^{\pi}\cong R\oplus R(-1)^{\oplus 20} \oplus R(-2) 
\]
as graded $R$-modules. Here, $\pi:\underline{\widehat{W}}(R)\twoheadrightarrow R$ is the composite of the two projection maps $\underline{\widehat{W}}(R)\twoheadrightarrow \widehat{W}(R)$ and $\widehat{W}(R)\twoheadrightarrow R$, and we put $M^{\pi} := M\otimes_{\underline{\widehat{W}}(R),\pi} R$.
\end{dfn}

The following result is proved by Langer-Zink \cite[Proposition 19]{lz2}
(see also \cite[Section 8.1]{lau2}).

\begin{prop}
Let $X$ be a $K3$ surface of over $R$. Then $H^2_{\crys}(X/\widehat{W}(R))$ is canonically equipped with $K3$ display structure. To be more precise, there exists a canonical $K3$ display \begin{align*}
  (\underline{H}^2_{\crys}(X/\widehat{W}(R)), (,))  
\end{align*} satisfying the following conditions:
\begin{enumerate}
    \item The degree $0$ part $(\underline{H}^2_{\crys}(X/\widehat{W}(R)))_0$ coincides with $H^2_{\crys}(X/\widehat{W}(R))$, and 
 \begin{align*}
    (,)_0:H^2_{\crys}(X/\widehat{W}(R))\times H^2_{\crys}(X/\widehat{W}(R))\to \widehat{W}(R) 
 \end{align*}
 is the pairing defined by Poincar\'e duality.
    \item For $i=1,2$, let $M^i$ be the image of the composite of the following maps:
    \begin{align*}
    (\underline{H}^2_{\crys}(X/\widehat{W}(R)))_i\xrightarrow{\times t^i} H^2_{\crys}(X/\widehat{W}(R))\twoheadrightarrow H^2_{\dR}(X/R).
    \end{align*}
 Then
 \begin{align*}
    H^2_{\dR}(X/R)\supset M^1\supset M^2 
 \end{align*}
 is the Hodge filtration.
 \item The $K3$ display $(\underline{H}^2_{\crys}(X/\widehat{W}(R)), (,))$ is compatible with base change with respect to $R$.
\end{enumerate}
\end{prop}
 
 \begin{proof}
  We briefly give a sketch of the proof for readers' convenience.
  Let $X_0:=X\otimes_{R}k$. Let $\sX\to \Spf(W(k)[[t_1,\ldots,t_{20}]])$ be the universal deformation of $X_0$. We put $A:=W(k)[[t_1,\ldots,t_{20}]]$ and $A_n:=A/(t_1^n,\ldots,t_{20}^n)$. Let $\sigma$ be a Frobenius lift such that $\sigma(t_i)=t_{i}^p$ for all $i$ and such that $\sigma$ coincides with the Witt vector Frobenius on $W(k)$. The Frobenius lift $\sigma$ defines the tautological frames $\underline{A}$ and $\underline{A_n}$. We have the unique morphism of local $W(k)$-algebras $A\to R, (t_i\mapsto x_i)$ corresponding to $X/R$. The morphism of local $W(k)$-algebras $A\to \widehat{W}(R)$ which sends $t_i$ to $[x_i]$ for all $i$ induces a morphism of frames $\underline{A}\to \underline{\widehat{W}}(R)$. Here, $[x_i]$ denotes the Teichm\"uller lift of $x_i$. Since $R$ is an Artin local ring, the ring morphism $A\to R$ (resp.\ the morphism of frames $\underline{A}\to \underline{\widehat{W}}(R)$) factors as $A\twoheadrightarrow A_m\to R$ (resp.\ $\underline{A}\to \underline{A_m}\to \underline{\widehat{W}}(R)$) for some $m\geq1$. Then, by taking base change of the $\underline{A_m}$-display $\underline{H}^2_{\dR}((\sX\otimes_A A_m)/A_m)$ (defined in Subsection \ref{subsection:2.5}) and a symmetric perfect pairing
  \[
    \underline{H}^2_{\dR}((\sX\otimes_A A_m)/A_m)\times \underline{H}^2_{\dR}((\sX\otimes_A A_m)/A_m)\to \underline{A_m}(-2) 
  \]
  by the morphism of frames $\underline{A_m}\to \underline{\widehat{W}}(R)$, we get the desired $K3$ display.
 \end{proof}
 The following theorem is a display-theoretic deformation theory for $K3$ surfaces proved by Lau \cite{lau2}. (When $X$ is ordinary, it is proved by Langer and Zink \cite{lz2}.)
 \begin{thm}[Lau] \label{dispdef}
 Let $S\twoheadrightarrow R$ be a surjection of Artin local rings with perfect residue field of characteristic $p\geq3$ and $X$ be a $K3$ surface over $R$. Then 
 \begin{align*}
    X'/S\mapsto \underline{H}^2_{\crys}(X'/\widehat{W}(S)) 
 \end{align*} gives an equivalence of categories between deformations of $X$ to $K3$ surfaces over $S$ and deformations of $\underline{H}^2_{\crys}(X/\widehat{W}(R))$ to $K3$ displays over $\widehat{W}(S)$.
 \end{thm}

\begin{proof}
See \cite[Theorem 8.1.1]{lau2}.
\end{proof}

\section{Crystalline cohomology and the enlarged formal Brauer group}
\label{Section:3}
First, we recall some results in \cite[Section 3]{no}. Let $p$ be a prime number, $k$ be a perfect field of characteristic $p$, $R$ be an Artin local ring over $W(k)$ with residue field $k$, and $X/R$ be a $K3$ surface of finite height (i.e.\ a proper flat scheme over $R$ whose closed fiber $X_{0}$ is a $K3$ surface of finite height). To this, we can associate a formal $p$-divisible group $\widehat{\Br}_{X/R}$ which is called the \emph{formal Brauer group} of $X/R$ and a $p$-divisible group $\psi_{X/R}$ which is called the \emph{enlarged formal Brauer group}. We note that $\widehat{\Br}_{X/R}$ is the connected component of the identity of $\psi_{X/R}$. For a $p$-divisible group $G$ over $R$, let $\bD(G)$ be the (covariant) Dieudonn\'e crystal on $\NilCrys(\Spec(R)/\Spf(\bZ_p))$. Let $\bH^{2}_{\crys}(X/R)$ be the crystal on ${\Crys}(\Spec(R)/\Spf(\bZ_p))$ such that
\[
 \bH^{2}_{\crys}(X/R)(S\twoheadrightarrow R)=H^{2}_{\crys}(X/S)
\]
for any $(S\twoheadrightarrow R)\in \Crys(\Spec(R)/\Spf(\bZ_p))$.
By \cite[Theorem 3.16]{no}, there is a natural morphism of $F$-crystals
\begin{align*}
   \rho:\bD(\psi_{X/R})\to \bH^{2}_{\crys}(X/R).
\end{align*}

By the composition with the inclusion $\bD(\widehat{\Br}_{X/R})\hookrightarrow \bD(\psi_{X/R})$, we get a morphism $\bD(\widehat{\Br}_{X/R})\to \bH^{2}_{\crys}(X/R)$. Taking the dual, we get a morphism of $F$-crystals \begin{align*}
 \bH^{2}_{\crys}(X/R)\to \bD(\widehat{\Br}^{\ast}_{X/R})(-1)   
\end{align*} because $\bH^{2}_{\crys}(X/R)^{\ast}=\bH^{2}_{\crys}(X/R)(2)$ and $\bD(\widehat{\Br}_{X/R})^{\ast}=\bD(\widehat{\Br}^{\ast}_{X/R})(1)$.

\begin{thm} \label{3.2}
There is the following exact sequence of $F$-crystals:
\[
    0\to \bD(\psi_{X/R})\xrightarrow{\rho} \bH^{2}_{\crys}(X/R)\to \bD(\widehat{\Br}^{\ast}_{X/R})(-1)\to 0.
\]
Moreover, by evaluating on the trivial PD thickening $R\xrightarrow{\id} R$, we get the following commutative diagram:
\[
\xymatrix{
0 \ar[r] & 0 \ar[r] \ar[d] & F^{2}H^{2}_{\dR}(X/R) \ar[r] \ar@{^{(}->}[d] & F^{1}\bD(\widehat{\Br}^{\ast}_{X/R})(-1)_{R} \ar[r] \ar@{^{(}->}[d] 
& 0
\\
0 \ar[r] & F^{1}\bD(\psi_{X/R})_{R} \ar[r] \ar@{^{(}->}[d] & F^{1}H^{2}_{\dR}(X/R) \ar[r] \ar@{^{(}->}[d] & \bD(\widehat{\Br}^{\ast}_{X/R})(-1)_{R}
\ar[r] \ar@{=}[d] & 0
\\
0 \ar[r] & \bD(\psi_{X/R})_{R} \ar[r]^{\rho} & H^{2}_{\dR}(X/R) \ar[r] & \bD(\widehat{\Br}^{\ast}_{X/R})(-1)_{R} \ar[r] & 0
}
\]
In this diagram, all horizontal rows are exact. 
\end{thm}
\begin{proof}
See \cite[Theorem 3.20]{no}.
\end{proof}

In this section, we prove a display-theoretic analogue of Theorem \ref{3.2}. By evaluating on the PD thickening $\widehat{W}(R)\twoheadrightarrow R$, we get a morphism of $\widehat{W}(R)$-modules 
\begin{align*}
 \hat{\rho}: \bD(\psi_{X/R})_{(\widehat{W}(R)\to R)}\to H^{2}_{\crys}(X/\widehat{W}(R)).   
\end{align*} We want to construct a morphism of $\underline{\widehat{W}}(R)$-displays $\underline{D}(\psi_{X/R})\to \underline{H}^{2}_{\crys}(X/\widehat{W}(R))$ which induces $\hat{\rho}$ on the underlying $\widehat{W}(R)$-modules. In order to do this, we work on $p$-torsion free bases.

In the following, we assume $p\geq3$. Let $\sX$$ \to$ Spf $W(k)[[t_{1},...,t_{20}]]$ be the universal deformation of $X_{0}$. We put $A:= W(k)[[t_{1},...,t_{20}]]$, $A_{n}:= A/(t_{1}^{n},...,t_{20}^{n})$, and $\sX_{n}:= \sX \otimes_A A_{n}$. Let $\sigma:A\to A$ be the endomorphism such that $\sigma(t_{i})=t_{i}^{p}$ and such that $\sigma$ is the Frobenius on $W(k)$. By evaluating $\rho$ on the trivial PD thickening $A_{n}\xrightarrow{\id} A_{n}$, we get a morphism of $A_{n}$-modules
\[ \rho_{A_{n}}: \bD(\psi_{\sX_{n}/A_{n}})_{A_{n}}\to H^{2}_{\dR}(X_{n}/A_{n}). \]
Here $\bD(\psi_{\sX_{n}/A_{n}})_{A_{n}}$ is the underlying $A_{n}$-module of the Dieudonn\'e display associated to $\psi_{\sX_{n}/A_{n}}$, and $H^{2}_{\dR}(\sX_{n}/A_{n})$ is  the underlying $A_{n}$-module of the $K3$ display associated to the $K3$ surface $\sX_{n}/A_{n}$.

\begin{prop} \label{3.3}
The map $\rho_{A_{n}}$ defines a morphism of $\underline{A_{n}}$-displays
\begin{align*}
 \underline{\rho_{A_{n}}} : \underline{D}(\psi_{\sX_{n}/A_{n}})\to \underline{H}^{2}_{\dR}(\sX_{n}/A_{n}).   
\end{align*}
\end{prop}

\begin{proof}
It suffices to prove that $\rho_{A_{n}}$ commutes with $F$ and preserves filtrations. The former is clear because $\rho$ is a morphism of $F$-crystals. We prove the latter assertion. Since $\underline{D}(\psi_{\sX_{n}/A_{n}})$ is a 1-display, it suffices to prove that
$\rho_{A_{n}}(F^{1}D(\psi_{\sX_{n}/A_{n}}))$ is contained in $F^{1}H^{2}_{\dR}(\sX_{n}/A_{n})$. This assertion follows from the commutative diagram in Theorem \ref{3.2}.
\end{proof}

Composing $\rho_{A_{n}}$ with $\underline{D}(\widehat{\Br}_{\sX_{n}/A_{n}})\hookrightarrow \underline{D}(\psi_{\sX_{n}/A_{n}})$, we get a morphism of $\underline{A_n}$-displays \ $\underline{D}(\widehat{\Br}_{\sX_{n}/A_{n}})\to \underline{H}^{2}_{\dR}(\sX_{n}/A_{n})$. By taking the dual of this map, we get a morphism of $\underline{A_n}$-displays \[ \underline{H}^{2}_{\dR}(\sX_{n}/A_{n})\to \underline{D}(\widehat{\Br}^{\ast}_{\sX_{n}/A_{n}})(-1) \]
because
$\underline{D}(\widehat{\Br}_{\sX_{n}/A_{n}})^{\ast}\cong \underline{D}(\widehat{\Br}^{\ast}_{\sX_{n}/A_{n}})(1)$
and 
$(\underline{H}^{2}_{\dR}(\sX_{n}/A_{n}))^{\ast}\cong \underline{H}^{2}_{\dR}(\sX_{n}/A_{n})(2)$.

\begin{thm} \label{3.4}
There is the following exact sequence of $\underline{A_n}$-displays:
\begin{align*}
 0\to \underline{D}(\psi_{\sX_{n}/A_{n}})\xrightarrow{\underline{\rho_{A_{n}}}} \underline{H}^{2}_{\dR}(\sX_{n}/A_{n})\to \underline{D}(\widehat{\Br}^{\ast}_{\sX_{n}/A_{n}})(-1)\to 0  
\end{align*}
\end{thm}

\begin{proof}
It suffices to show that the following three sequences of  $A_{n}$-modules are exact.
\begin{align*}
 0 &\to D(\psi_{\sX_{n}/A_{n}})\xrightarrow{\rho_{A_{n}}} H^{2}_{\dR}(\sX_{n}/A_{n})\to D(\widehat{\Br}^{\ast}_{\sX_{n}/A_{n}}) \to 0 \\
 0 &\to F^{1}D(\psi_{\sX_{n}/A_{n}})\xrightarrow{\rho_{A_{n}}} F^{1}H^{2}_{\dR}(\sX_{n}/A_{n})\to D(\widehat{\Br}^{\ast}_{\sX_{n}/A_{n}}) \to 0 \\
 0 &\to F^{2}D(\psi_{\sX_{n}/A_{n}})\xrightarrow{\rho_{A_{n}}} F^{2}H^{2}_{\dR}(\sX_{n}/A_{n})\to F^{1}D(\widehat{\Br}^{\ast}_{\sX_{n}/A_{n}}) \to 0 
\end{align*}
The exactness of all sequences follows from Theorem \ref{3.2}.
\end{proof}
 
\begin{cor} \label{3.5}
Let $X/R$, $\psi_{X/R}$, and $\widehat{\Br}_{X/R}$ be the same as the beginning of this section. Then there is the following exact sequence of $\underline{\widehat{W}}(R)$-displays:
\begin{align*}
  0\to \underline{D}(\psi_{X/R})\to \underline{H}^{2}_{\crys}(X/\widehat{W}(R))\to \underline{D}(\widehat{\Br}^{\ast}_{X/R})(-1)\to 0.  
\end{align*}
Moreover, this sequence is compatible with base change with respect to $R$.
\end{cor}

\begin{proof}
Let $A\to R,(t_i\mapsto x_i)$ be the morphism which corresponds to the deformation $X/R$. We define $\phi: A\to \widehat{W}(R)$ by $\phi(t_i)=[x_i]$, where $[x_i]$ denotes the Teichm\"uller lift of $x_i$. Then, for a sufficiently large $n$, the map $\phi$ induces a morphism of frames $\underline{A_{n}}\to \underline{\widehat{W}}(R)$. Taking the base change of the sequence in Theorem \ref{3.4} by this morphism, we get the desired exact sequence.
\end{proof}

\section{A constrution of quasi-canonical liftings}
\label{Section:4}
In this section, we prove the main result of this paper.

\begin{thm} \label{4.1}
Let $X_{0}$ be a $K3$ surface of finite height $h<\infty$ over a finite field $k$ of characteristic $p\geq 3$. Then there exist a totally ramified finite extension $V/W(k)$ of degree $h$ and a quasi-canonical lifting $X/V$ of $X_{0}/k$.
\end{thm}

\begin{proof}
We follow the strategy of Nygaard-Ogus \cite{no}.
Instead of the crystalline deformation theory used in \cite{no},
we use the display-theoretic deformation theory of $K3$ surfaces.

We put $k=\bF_{p^m}$. Let $\widehat{\Br}_{X_0/k}$ (resp.\ $\psi_{X_0/k}$) be the formal Brauer group (resp.\ the enlarged formal Brauer group) of $X_0/k$. Let $\Fr:\widehat{\Br}_{X_{0}/k}\to \widehat{\Br}^{(p)}_{X_{0}/k}$ be the relative Frobenius morphism. By \cite[Theorem 24.2.6, Theorem 24.3.4]{haz}, there exists a totally ramified finite extension $V/W(k)$ of degree $h$, a lifting $G$ of $\widehat{\Br}_{X_{0}/k}$ over $V$, and an endomorphism $F:G\to G$ which lifts $\Fr^{m}: \widehat{\Br}_{X_{0}/k}\to \widehat{\Br}_{X_{0}/k}$. Let $\pi$ be a prime element of $V$.

For each $n \geq 1$, we put $G_{n}:= G\otimes_V V/\pi^{n+1}$. Let $H_{n}$ denote the unique lifting of the \'etale part of $\psi_{X_{0}/k}$ over $V/\pi^{n+1}$. We define a $\underline{\widehat{W}}(V/\pi^{n+1})$-display $\underline{K}(G_n)$ as 
\begin{align*}
 \underline{K}(G_n):=\underline{D}(G_n)\oplus \underline{D}(H_{n})\oplus \underline{D}(G_n^{\ast})(-1). 
\end{align*} Taking base change by $\underline{\widehat{W}}(V/\pi^{n+1})\to \underline{W}(k)$, the right hand side gives the slope decomposition of $\underline{H}^{2}_{\crys}(X_{0}/W(k))$. We have a canonical perfect pairing
\begin{align*}
 \underline{D}(G_n)\times \underline{D}(G_{n}^{\ast})(-1)\to \underline{\widehat{W}}(V/\pi^{n+1})(-2)  
\end{align*} and a perfect pairing 
\begin{align*}
 \underline{D}(H_{n})\times \underline{D}(H_{n})\to \underline{\widehat{W}}(V/\pi^{n+1})(-2)  
\end{align*} which is the unique lifting of the pairing on the slope one part of $\underline{H}^{2}_{\crys}(X_{0}/W(k))$ (\cite[Lemma 42]{zi1}). These pairings define a K3 display structure on $\underline{K}(G_n)$.

Let $\sX$ be a formal lifting which corresponds to the inverse system of $K3$ displays $\{\underline{K}(G_n)\}$ (Theorem \ref{dispdef}). 
Let $X_{n}:= \sX \otimes_V  V/\pi^{n+1}$. Then $X_n$ is a $K3$ surface over $V/\pi^{n+1}$ corresponding to the $K3$ display $\underline{K}(G_n)$. We consider the following morphisms of $\underline{\widehat{W}}(V/\pi^{n+1})$-displays:
\begin{align*}
  \underline{D}(G_n^{\ast})(-1)\hookrightarrow \underline{K}(G_n)=\underline{H}^{2}_{\crys}(X_n/\widehat{W}(V/\pi^{n+1}))\twoheadrightarrow \underline{D}(\widehat{\Br}^{\ast}_{X_n/(V/\pi^{n+1})})(-1),  
\end{align*} where the last morphism is the one in Corollary \ref{3.5}.
The composite of the above morphisms is an isomorphism of displays because it becomes an isomorphism after the base change by $\underline{\widehat{W}}(V/\pi^{n+1})\to \underline{W}(k)$. Then we have the following diagram.
\[
\xymatrix{
0 \ar[r] & \underline{D}(\psi_{X_n/(V/\pi^{n+1})}) \ar[r] \ar@{.>}[d] & \underline{H}^{2}_{\crys}(X_n/\widehat{W}(V/\pi^{n+1})) \ar[r] \ar@{=}[d] & \underline{D}(\widehat{\Br}^{\ast}_{X_n/(V/\pi^{n+1})})(-1) \ar[r] \ar[d]^-{\simeq} & 0
\\ 0 \ar[r] & \underline{D}(G_n)\oplus \underline{D}(H_n) \ar[r] & \underline{K}(G_n) \ar[r] & \underline{D}(G_n^{\ast})(-1) \ar[r] & 0
}
\]
In this diagram, both rows are exact by Theorem \ref{3.4}. So we get a canonical isomorphism of displays \begin{align*}
  \underline{D}(\psi_{X_n/(V/\pi^{n+1})})\stackrel{\sim}{\to} \underline{D}(G_n)\oplus \underline{D}(H_n)=\underline{D}(G_n\oplus H_n).  
\end{align*}

By Theorem \ref{dieudonne theory}, this induces an isomorphism
\[ \psi_{X_n/(V/\pi^{n+1})}\stackrel{\sim}{\to} G_n\oplus H_{n} \]
of $p$-divisible groups over $V/\pi^{n+1}$. In particular, the enlarged formal Brauer group $\psi_{X_n/(V/\pi^{n+1})}$ splits into a direct sum of a connected part and an \'etale part. Since the map $\Pic(X_n)\to \Pic(X_0)$ is surjective by the argument in the proof of \cite[Proposition 1.8]{ny}, there is a proper smooth scheme $X$ over $V$ and a line bundle $\sL$ on $X$ such that $X\otimes_V V/\pi^{n+1}=X_n$ and $\sL|_{X_0}$ is ample and primitive.

By \cite[Definition 1.5 and Theorem 1.9]{no}, in order to prove that $X/V$ is a quasi-canonical lifting of $X_0/k$, it suffices to prove that there exists $\gamma\in W_{\crys}(\overline{K})$ of non-zero degree such that the action of $\gamma$ on $H^2_{\dR}(X/V)\otimes_V \overline{K}$ preserves the Hodge filtration. Let $H$ denotes the unique lifting of the \'etale part of $\psi_{X_0/k}$. By the definition of $\underline{K}(G_n)$,
\begin{align*}
    H^2_{\dR}(X/V)&=\bD(G)_V\oplus \bD(H)_V\oplus \bD(G^{\ast})(-1)_V
    \\ &\supset F^1\bD(G)_V\oplus \bD(H)_V\oplus \bD(G^{\ast})(-1)_V
    \\ &\supset F^1\bD(G^{\ast})(-1)_V
\end{align*} is the Hodge filtration on $H^2_{\dR}(X/V)$. Here $\overline{K}$ is an algebraic closure of the fraction field of $V$ and $W_{\crys}(\overline{K})$ is the crystalline Weil group (for the definition, see \cite[Definition 4.1]{bo2}). By \cite[Proposition 3.14]{bo2}, we have  $W_{\crys}(\overline{K})$-equivariant isomorphisms
\begin{align*}
    H^2_{\dR}(X/V)\otimes \overline{K}&=(\bD(G)_V\oplus \bD(H)_V\oplus \bD(G^{\ast})(-1)_V)\otimes_V \overline{K}
    \\ &\cong (\bD(\widehat{\Br}_{X_0/k})_{W(k)}\oplus \bD(H_0)_{W(k)}\oplus \bD(\widehat{\Br}_{X_0/k}^{\ast})(-1)_{W(k)})\otimes_{W(k)} \overline{K}
    \\ &\cong H^2_{\crys}(X_0/W(k))\otimes_{W(k)} \overline{K}.
\end{align*} Since the isomorphism 
\[
 \bD(G^{\ast})_V\otimes_V \overline{K}\cong \bD(G_{0}^{\ast})_{W(k)}\otimes_{W(k)} \overline{K}=\bD(\widehat{\Br}_{X_0/k}^{\ast})_{W(k)}\otimes_{W(k)} \overline{K}
\]
 is functorial in $G$, an element $\gamma\in W_{\crys}(\overline{K})$ with $\deg(\gamma)=m$ acts on $\bD(G^{\ast})_V\otimes_V \overline{K}$ as $\bD(F^{\ast})\otimes \gamma$. Therefore, the action of such $\gamma\in W_{\crys}(\overline{K})$ on $H^2_{\dR}(X/V)\otimes_V \overline{K}$ preserves the Hodge filtration. This completes the proof.
\end{proof}

\subsection*{Acknowledgements}
The author would like to thank his advisor, Tetsushi  Ito, for useful discussions and warm encouragement. The author would also  like to thank Kazuhiro Ito for helpful discussions.

\end{document}